\documentclass[11pt]{amsart}
\usepackage{pstricks}
\usepackage{color}
\theoremstyle{plain}
\newtheorem{thm}{Theorem}[section]
\newtheorem{theorem}[thm]{Theorem}

\newtheorem{lemma}[thm]{Lemma}

\newtheorem{proposition}[thm]{Proposition}
\theoremstyle{definition}
\newtheorem{remark}[thm]{Remark}

\newtheorem{definition}[thm]{Definition}

\numberwithin{equation}{section}
\newcommand{\sE}{{\mathcal E}}

\newcommand{\sO}{{\mathcal O}}

\newcommand{\sS}{{\mathcal S}}

\newcommand{\A}{{\mathbb A}}

\newcommand{\C}{{\mathbb C}}

\newcommand{\F}{{\mathbb F}}

\renewcommand{\P}{{\mathbb P}}
\newcommand{\Q}{{\mathbb Q}}
\newcommand{\R}{{\mathbb R}}

\newcommand{\Z}{{\mathbb Z}}

\title [Elliptic fibrations and maximal Salem degree]{Elliptic fibrations on K3 surfaces and Salem numbers of maximal degree}

\author{Xun Yu} 
\address{Center for Applied Mathematics, Tianjin University, 92 Weijin Road, Nankai District,
Tianjin 300072, P. R. China.
}
\email{xunyu@tju.edu.cn}

\begin{document}
\begin{abstract}  
We study the maximal Salem degree of automorphisms of K3 surfaces via elliptic fibrations. By generalizing  \cite{EOY14}, we establish a characterization of such maximum in terms of elliptic fibrations with infinite automorphism groups. As an application, we show that any supersingular K3 surface in odd characteristic has an automorphism the entropy of which is the natural logarithm of a Salem number of degree $22$.
\end{abstract}
\maketitle
\section{Introduction}\label{intro}

%A supersingular K3 surface is a K3 surface with Picard number $22$, the maximal possible value. Supersingular K3 surfaces have attracted a lot of attention in the past few decades, and, in many senses, such surfaces are the most special K3 surfaces.

Let $X$ be a K3 surface defined over an algebraically closed field $k$ of characteristic $p \ge 0$, that is, $X$ is a smooth projective surface defined over $k$ such that $H^1(X, \sO_X) = 0$ and the dualizing sheaf is trivial : $\omega_X \simeq \sO_X$. We denote by ${\rm NS}(X)$ the N\'eron-Severi group of $X$. 

By definition, the {\it entropy }of an automorphism $f:X\longrightarrow X$ is the logarithm of the maximal absolute value of the eigenvalues of the action $f^*$ on ${\rm NS}(X)$ induced by $f$. This definition is consistent to the topological entropy of automorphisms of smooth complex projective surfaces (\cite{ES13}). One knows that the entropy of any automorphism of $X$ is either $0$ or the logarithm of a Salem number, and we define {\it Salem degree of $f$} to be $0$ or the degree of the Salem number respectively  (see Section~\ref{ss:salem}).  The maximal Salem degree of automorphisms of $X$ is closely related to both complexity and richness of such automorphisms. Thus, a good understanding of such maximum is of interest. The main results of this note are Theorems \ref{thm:main2}, \ref{thm:main3}, and \ref{thm:main1}.

As in \cite{Ni14}, we say an elliptic fibration on $X$ is an elliptic fibration {\it with infinite automorphism group} if the set of all automorphisms of $X$ which preserve this fibration is infinite (see Section~\ref{ss:elliptic}). By generalizing  \cite{EOY14}, we establish a characterization of the maximal Salem degree of automorphisms of K3 surfaces in terms of elliptic fibrations with infinite automorphism groups:

\begin{theorem}\label{thm:main2}
Let $X$ be a K3 surface defined over an algebraically closed field of characteristic $p\ne 2, 3$.  Suppose ${\rm rk}(L_{\infty}(X))\ge 2$, where the sublattice $L_{\infty}(X)\subset {\rm NS}(X)$ is defined to be generated by all the elliptic fibrations with infinite automorphism groups. Let $d={\rm rk}(L_{\infty}(X))$. Then
\begin{itemize}
\item[1)] If $d$ is even,  ${\rm max}\{\text{Salem degree of } f|\; f\in {\rm Aut}(X)\}=d$;

\item[2)] If $d$ is odd,  ${\rm max}\{\text{Salem degree of } f|\; f\in {\rm Aut}(X)\}=d-1$.
\end{itemize}
\end{theorem}
 
  Using Theorem~\ref{thm:main2}, we find the following comparison type theorem:

\begin{theorem}\label{thm:main3}
Let $X$ and $Y$ be two K3 surfaces of the same Picard number defined over two algebraically closed fields $k$ and $k^{\prime}$ (${\rm char}(k),{\rm char}(k^{\prime})\ne 2,3$). Suppose  ${\rm NS}(X)$ is isometric to a sublattice of ${\rm NS}(Y)$, and suppose $Y$ has at least one elliptic fibration with infinite automorphism group. Then $${\rm max}\{\text{Salem degree of } f|\; f\in {\rm Aut}(X)\}\ge {\rm max}\{\text{Salem degree of } f|\; f\in {\rm Aut}(Y)\}.$$
\end{theorem}

\medskip

     Recall that a supersingular K3 surface is a K3 surface with Picard number $22$, the maximal possible value. In many senses, supersingular K3 surfaces are the most special K3 surfaces. An interesting application of Theorem  \ref{thm:main3} is the following:

\begin{theorem}\label{thm:main1}
Let $p$ be an odd prime. Let $X$ be a supersingular K3 surface defined over an algebraically closed field of characteristic $p$. Then there is an automorphism $f\in {\rm Aut}(X)$ the entropy of which is the logarithm of a Salem number of degree $22$.
\end{theorem}

In particular, those automorphisms are not geometrically liftable to characteristic $0$ (see \cite{EO15}, \cite{EOY14}).    Many people have studied supersingular K3 surface automorphisms of Salem degree $22$ in recent years  (\cite{BC16}, \cite{EO15}, \cite{EOY14}, \cite{Sh16}, \cite{Sch15}, \cite{Br15}).  More precisely, we summarize previously known results on such automorphisms as follows:

\begin{theorem}\label{thm:known}
 Let $p$ be a prime number. Let $X$ be a supersingular K3 surface defined over an algebraically closed field of characteristic $p$.  Let $\sigma(X)$ be the Artin invariant of $X$. Then, in the following cases, $X$ has an automorphism the entropy of which is the logarithm of a Salem number of degree $22$:
 \begin{itemize}
 
  \item[(1)] $\sigma(X)=1$ and $p=2$ (\cite{BC16});
  
  \item[(2)] $\sigma(X)=1$ and $p=3$  (\cite{EO15});
  
  \item[(3)]  $\sigma(X)=1$ and $p=11$ or $>13$ (\cite{EOY14});
  
  \item[(4)] $\sigma(X)=1$ and $p\in\{5,7,13\}$, or  $2\le \sigma(X)\le 9$ and $3\le p\le 7919$, or $ \sigma(X)=10$ and $3\le p\le 17389$ (\cite{Sh16}).
\end{itemize}
\end{theorem}

% For the cases (1) and (2) the proofs rely on the very detailed studies of the automorphism groups of the surfaces in question (\cite{DK02}, \cite{KS14}). For the case (4) the proof is computer aided and uses double plane involutions of supersingular K3 surfaces. For the case (3) the proof is abstract and   based on the theory of the Mordell-Weil groups of elliptic fibrations {\it with a section} on supersingular K3 surfaces with Artin invariant one (\cite{Shi13}). 
%

%In fact,  classification  of the N\'eron-Severi groups of supersingular K3 surfaces are known (\cite{RS78}). 

It is known that for any supersingular K3 surface $X$ in odd characteristic, the N\'eron-Severi group ${\rm NS}(X)$ is isometric to a sublattice of the N\'eron-Severi group of a supersingular K3 surface of Artin invariant one in the same characteristic (\cite{RS78}, cf.~\cite{Li15}). Thus, by combining Theorem~\ref{thm:main3} and Theorem~\ref{thm:known}, we can  prove Theorem~\ref{thm:main1}. The author was recently informed that Simon Brandhorst found an alternative proof of Theorems \ref{thm:main3} and \ref{thm:main1}.

In Section~\ref{ss:E}, following \cite{Ni14}, we introduce the notion of the {\it exceptional sublattice} $E({\rm NS}(X))\subset {\rm NS}(X)$. For an elliptic K3 surface $X$, the exceptional sublattice, the set of all the elliptic fibrations with infinitie automorphism groups and the maximal Salem degree of automorphisms of $X$ are closely related to each other (Theorems~\ref{thm:main2}, \ref{thm:E}, \ref{thm:evenPicard}).   It is hoped that the exploration of such relationships in this note will find interesting applications in future study of K3 surfaces and other related topics. 

\medskip

\noindent
{\bf Acknowledgement.}  The author would like to thank Professor Keiji Oguiso for valuable discussions and comments.\\[.2cm]

\section{Preliminaries on K3 surfaces and elliptic fibrations} \label{sec:notations}
\noindent 

In this section, we fix notations on K3 surfaces and elliptic fibrations.

\subsection{K3 surfaces} \label{ss:k3}

Let $X$ be a K3 surface defined over an algebraically closed field $k$ of characteristic $p \ge 0$. We denote by ${\rm NS}(X)$ the N\'eron-Severi group of $X$. Then the Picard group
${\rm Pic}\, (X) $ is isomorphic to  N\'eron-Severi group  ${\rm NS}(X)$, which is a free $\Z$-module of finite rank. The rank of ${\rm NS}(X)$ is called the Picard number of $X$ and is denoted by $\rho(X)$.  It is at least $1$ as $X$ is assumed to be projective, and at most $22$, the second $\ell$-adic Betti number. In characteristic $0$, it is at most $20$ by Hodge theory.
The intersection form $(*, **)$ on ${\rm NS}(X)$ is of signature $(1, \rho(X) -1)$ and $({\rm NS}(X), (*,**))$ is then an even  hyperbolic lattice. The dual $\Z$-module ${\rm NS}(X)^* := {\rm Hom}_{\Z}({\rm NS}(X), \Z)$ is regarded as a $\Z$-submodule of ${\rm NS}(X) \otimes \Q$, containing ${\rm NS}(X)$ through the intersection form $(*, **)$ which is non-degenerate. The quotient module ${\rm NS}(X)^*/{\rm NS}(X)$ is called the discriminant group of $X$.

%The surface $X$ is called supersingular  if   $\rho(X) = 22$, the maximum possible value. 
%(As the Tate conjecture is not yet  proven  for $p=2$, one should rather say  Shioda  supersingular in this case, but we 
%consider only supersingular K3 surfaces in odd characteristic  in this note). 

 Artin \cite{Ar74}  proved that the discriminant group ${\rm NS}(X)^*/{\rm NS}(X)$ of a supersingular K3 surface $X$ is $p$-elementary, more precisely,  as an abelian group
$${\rm NS}(X)^*/{\rm NS}(X) \simeq (\Z/p)^{2\sigma(X)}$$
where $\sigma(X)$ is an integer such that $1 \le \sigma(X) \le 10$. The integer $\sigma(X)$ is called the Artin invariant of $X$. Let $\sigma$ be an integer such that $1 \le \sigma \le 10$. Then the supersingular K3 surfaces over $k$ with   $\sigma(X) \le \sigma$   form $(\sigma -1)$-dimensional family over $k$. 
There is, up to isomorphism, a unique Artin invariant $1$ supersingular K3 surface $X(p)$  in each positive characteristic $p>0$. When $p\ge 3$, the uniqueness of $X(p)$ in particular shows that $X(p) \simeq {\rm Km}\, (E \times_{\F_p} E)$ for any supersingular elliptic curve $E$ over $\F_p$ (\cite{Ogu79}, \cite{Shi75}).

\subsection{Elliptic fibrations on K3 surfaces}  \label{ss:elliptic}

Let $X$ be a K3 surface defined over an algebraically closed field $k$ of characteristic $p\ne 2,3$. According to Piatetsky-Shapiro and Shafarevich \cite{PS71}, elliptic fibrations on $X$ are in one-to-one correspondence with primitive isotropic nef elements $e\in {\rm NS}(X)$. That is, $e\ne 0$, $e^{2}=0$, $e/n\in {\rm NS}(X)$ only for integers $n=\pm 1$, $e\cdot D\ge 0$ for any effective divisor $D$ on $X$. For such $e\in {\rm NS}(X)$, the complete linear system $|e|$ is one dimensional without base points, and it gives an elliptic fbration $|e|: X\longrightarrow \P^{1}$, that is, the general fiber is an elliptic curve.

 For any $c\in {\rm NS}(X)$, we set $${\rm Aut}(X)_{c}:=\{f\in {\rm Aut}(X)| f^{*}(c)=c\}.$$ We introduce some notations related to elliptic fibrations on $X$:

\medskip

$\sE(X):=\{e\in {\rm NS}(X) | e \text{ is primitive, isotropic, and nef }\}$;

$\sE_{\infty}(X):=\{e\in \sE(X) | {\rm Aut}(X)_{e} \text{ is an infinite group}\}$, and we say an elliptic fibration on $X$ is an elliptic fibration {\it with infinite automorphism group} if the corresponding class of this fibration is in $\sE_{\infty}(X)$;

%$\sE_{\infty}^{J}:=\{e\in \sE| e \text{ corresponds to an elliptic fibration with a section and } rank(MW(e))>0\}$;

The sublattice $L_{\infty}(X)\subset {\rm NS}(X)$ is generated by all the elements in $\sE_{\infty}(X)$.

\medskip

%$L_{\infty}^{J}:=$ the sub-lattice of ${\rm NS}(X)$ generated by all the elements in $\sE_{\infty}^{J}$;
%
%Obviously, $\sE_{\infty}^{J}\subseteq \sE_{\infty}$, and $L_{\infty}^{J}\subseteq L_{\infty}\subseteq S_{X}$.

  For any nonzero isotropic element $c\in {\rm NS}(X)$, the sublattice $(c^{\perp})^{(2)}\subset c^{\perp}$ is generated by $c$ and by all elements with square $(-2)$ in $c^{\perp}$, where $c^{\perp}$ is the orthogonal complement to $c$ in ${\rm NS}(X)$. The following lemma gives a test of an elliptic fibration with infinite automorphism group:

\begin{lemma}\label{lem:einf}
Let $e\in \sE(X)$. Then $e\in \sE_{\infty}(X)$  if and only if ${\rm rk}(e^{\perp})-{\rm rk}((e^{\perp})^{(2)})>0$.
\end{lemma}

\begin{proof}
We use ${\rm Amp}(X)\subset {\rm NS}(X)\otimes \R$ to denote the ample cone of $X$. Let $$A({\rm NS}(X)):=\{f\in {\rm O}({\rm NS}(X))| f({\rm Amp}(X))={\rm Amp}(X) \}.$$ Let $A({\rm NS}(X))_{e}\subset A({\rm NS}(X))$ be the stabilizer subgroup of $e$. By \cite[Theorem 6.1]{LM11}, the natural map ${\rm Aut}(X)\longrightarrow A({\rm NS}(X))$ has finite kernel and cokerel. Then ${\rm Aut}(X)_{e}$ and $A({\rm NS}(X))_{e}$ are isomorphic up to finite groups. By \cite[Corollary 1.5.4]{Ni83} (notice that the proof of \cite[Corollary 1.5.4]{Ni83} works for even hyperbolic lattices, in particular, ${\rm NS}(X)$, and is valid in any characteristic), $A({\rm NS}(X))_{e}$ is infinite if and only if ${\rm rk}(e^{\perp})-{\rm rk}((e^{\perp})^{(2)})>0$. Thus, ${\rm Aut}(X)_{e}$ is infinite if and only if ${\rm rk}(e^{\perp})-{\rm rk}((e^{\perp})^{(2)})>0$.

\end{proof}

Let $\phi: {\rm Aut}(X)\longrightarrow {\rm O}({\rm NS}(X))$ be the natural map which sends any $f\in {\rm Aut}(X)$ to the induced isometry $f^{*}$ of ${\rm NS}(X)$. By definition of $L_{\infty}(X)$, it is clear that $L_{\infty}(X)$ is an ${\rm Aut}(X)$-stable sublattice of ${\rm NS}(X)$ (i.e., for all $f\in {\rm Aut}(X)$, $f^{*}(L_{\infty}(X))=L_{\infty}(X)$). Thus $\phi$ naturally induces another map $$\psi: {\rm Aut}(X)\longrightarrow {\rm O}(L_{\infty}(X))\times {\rm O}(L_{\infty}(X)^{\bot})$$ such that, for any $f\in {\rm Aut}(X)$, $\psi(f)=(f^{*}|_{L_{\infty}(X)}, f^{*}|_{L_{\infty}(X)^{\bot}})$. Let $$\pi: {\rm O}(L_{\infty}(X))\times {\rm O}(L_{\infty}(X)^{\bot})\longrightarrow {\rm O}(L_{\infty}(X))$$ be the natural projection map.

\begin{lemma}\label{lem:Kerfinite}
Suppose $X$ has at least two different elliptic fibrations with infinite automorphism groups, i.e., ${\rm rk}(L_{\infty}(X))\ge 2$. Then  both ${\rm Ker}(\psi)$ and ${\rm Ker}(\pi\circ \psi)$ are finite groups.
\end{lemma}

\begin{proof}

Let $d={\rm rk}(L_{\infty}(X))$. Since $d\ge 2$, it follows that $L_{\infty}(X)$ is a hyperbolic lattice of signature $(1, d-1)$. Then the orthogonal complement $L_{\infty}(X)^{\bot}\subset {\rm NS}(X)$ of $L_{\infty}(X)$ is a negative definite lattice of rank $\rho(X)-d$. By \cite[Theorem 6.1]{LM11}, ${\rm Ker}(\phi)$ is a finite group. Note that $L_{\infty}(X)\oplus L_{\infty}(X)^{\bot}$ is a sublattice of ${\rm NS}(X)$ of finite index. Thus, ${\rm Ker}(\psi)$ is equal to ${\rm Ker}(\phi)$ and hence is also a finite group. The group ${\rm O}(L_{\infty}(X)^{\bot})$ is finite because $L_{\infty}(X)^{\bot}$ is negative definite. Since the quotient ${\rm Ker}(\pi\circ \psi)/{\rm Ker}(\psi)$ is isomorphic to a subgroup of ${\rm O}(L_{\infty}(X)^{\bot})$, it follows that ${\rm Ker}(\pi\circ \psi)/{\rm Ker}(\psi)$ is finite. Then ${\rm Ker}(\pi\circ \psi)$ is also finite.
\end{proof}

\section{Preliminaries on Salem numbers and entropy}  \label{ss:salem}
\label{sec:notations2}
\noindent 
In this section,   we
  recall the definition of entropy and Salem numbers from \cite{EOY14} and  the references therein. 

\medskip

In what follows, $L = (\Z^{1+t} , (*,**))$  is a hyperbolic lattice, i.e., a pair  consisting of a free $\Z$-module of rank $1+t$ and a $\Z$-valued symmetric bilinear form $( \ , \ )$ of $\Z^{1+t}$ of signature 
$(1, t)$ with $t > 0$. For any ring $K$, one denotes the scalar extension of $L$ to $K$ by $L_K$. We fix some notations as follows:

 ${\rm O}(L):=$ the orthogonal group  of the quadratic lattice $L$;
 
 ${\rm SO}(L):=\{g\in {\rm O}(L) | {\rm det}(g)=1 \}$;

$P:=\{x \in L_{\R} |  \ (x^2) > 0\}$, notice that $P$ consists of two connected components, say $\pm C$;

${\rm O}^{+}(L_{\R}) := \{ g \in {\rm O}(L)(\R) | \ g(C) = C \}$;

${\rm O}^{+}(L) := {\rm O}(L) (\Z) \cap {\rm O}^+(L_{\R})$;

${\rm SO}^{+}(L) := {\rm O}(L) (\Z) \cap {\rm O}^+(L_{\R}) \cap {\rm SO}(L)(\R).$

\begin{definition}\label{def:Salem}
We call a polynomial $P(x) \in \Z [x]$ a {\it Salem polynomial} if it is irreducible, monic,  of even degree $2d \ge 2$ and the complex zeroes of $P(x)$ are of the form ($1 \le i \le d-1$):
$$a > 1\,\, ,\,\, 0 < a^{-1} < 1\,\, ,\,\, \alpha_i, \overline{\alpha}_i \in S^1 := 
\{z \in \C\, \vert\, \vert z \vert = 1\} \setminus \{\pm 1\}\,\, .$$ A {\it Salem number} is the real root $>1$ of a Salem polynomial.

\end{definition}

\begin{proposition}\label{salem} Let $f \in {\rm O}^+(L)$. Then, the characteristic polynomial of $f$ is the product of cyclotomic polynomials and at most one Salem polynomial counted with multiplicities. 
\end{proposition}    
\begin{proof} As mentioned in \cite[Prop.~3.1]{EO15}, this is well-known. See 
\cite{Mc02}, \cite{Og10}.
\end{proof} 
\begin{definition}\label{entropy}  \begin{itemize}
\item[i)]
 For $f$ as in Proposition~\ref{salem}, we define the entropy 
$h(f)$ 
of $f$ by
$$h(f) = \log ( r(f) ) \ge 0\,\, ,$$
where $r(f)$ is the spectral radius of $f$, that is the maximum of the absolute values of the  complex
 eigenvalues  of $f$ acting on $L$. 
\item[ii)]
For a smooth projective surface $S$ and an automorphism $f$ on it,  one defines the entropy   $h(f)$  by
$$h(f) = \log r(f^* | {\rm NS}(X)) $$ 
where $f^*$ is the action on ${\rm NS}(X)$ induced by $f$. 
 \end{itemize}
\end{definition}
This definition is consistent to the topological entropy of automorphisms of smooth complex projective surfaces (\cite{ES13}).

\begin{definition}\label{def:Salemdegree}
For a K3 surface $X$ and an automorphism $f$ on it,  one defines the {\it Salem degree} of $f$ by

$\text{Salem degree of } f = \begin{cases} 0 &\mbox{if } \text{the entropy } h(f) \text{ is } 0, \\ 
d & \mbox{if } \text{the entropy } h(f) \text{ is log of a Salem number of degree } d. \end{cases} $

\end{definition}

\medskip

\begin{remark}
In order to study the maximal Salem degree of automorphisms of any K3 surface $X$, we will take $L$ to be the sublattice $L_{\infty}(X)$ of the N\'eron-Severi group ${\rm NS}(X)$ and $C$ to be  the connected component of $P$ the closure of which contains all elements in $\sE_{\infty}(X)$. 

\end{remark}

\section{Two observations from group theory} 
\label{sec:grouptheory}
\noindent
In this section, we shall prove  Theorem~\ref{thm:BlancCantat} and Theorem~\ref{thm:StableSublattice} which are  generalizations of \cite[Theorem 4.1]{EOY14} and \cite[Theorem 4.6]{EOY14}.

%  They are  crucial for our main theorems \ref{thm:main1}, \ref{thm:main2}, and \ref{thm:main3}.

\begin{thm} \label{thm:BlancCantat}
Let $L$ be a hyperbolic lattice of rank $d\geq2$ and $G \subset {\rm SO}^{+}(L)$ be  a subgroup. Assume that $G$ has no $G$-stable $\R$-linear subspace of $L_{\R}$ other than $\{0\}$ and $L_{\R}$. Then 
\begin{itemize}
\item[i)] If $d$ is even, then there is  an element $g \in G$ such that the  characteristic polynomial of $g$  is a Salem polynomial of degree $d$;

\item[ii)] If $d$ is odd, then there is  an element $g \in G$ such that the  characteristic polynomial of $g$  is $(t-1)s(t)$, where $s(t)$ is a Salem polynomial of degree $d-1$.

\end{itemize}
\end{thm} 
The case i) is just \cite[Theorem 5.1]{EOY14}. The proof of the case ii)  is completely similar to that of the case i), and we will sketch it below:
% Like the proof of \cite[Theorem 5.1]{EOY14}, the proof below  mimics \cite[p.~426-427]{BC16}, where the authors handle the case of $L={\rm NS}(X(2))$,  and deduces the statement from \cite[Prop.~1]{BH04}. 
% 
%%We slightly clarify  their argument to make it fit with  Theorem~\ref{thm:BlancCantat}.

\begin{proof} 
%Recall  \cite[Prop.~1]{BH04}:

%\begin{thm} \label{thm:delaHarpe}
%Let $G \subset {\rm O}(L)(\R)$ be an abstract  subgroup. Assume that $G$ has no $G$-stable $\R$-linear subspace of $L_{\R}$ other than $\{0\}$ and $L_{\R}$. Then the Zariski closure of $G$ in ${\rm O}(L_{\R})$ contains  ${\rm SO}(L_{\R})$.  In particular, if  in addition  $G\subset {\rm SO}(L)(\R)$, then its Zariksi closure in ${\rm SO}(L_{\R})$ is 
%${\rm SO}(L_{\R})$.
%\end{thm}
%
%As is well known, if $L$ is any non-degenerate quadratic lattice, and $g \in {\rm SO}(L)(\Z)$, then if the rank of $L$ is odd, $1$ is an eigenvalue of $g$.  Indeed, if $Q, M$ are the matrices of the form and $g$ in a chosen basis, then 
%$^tMQ(M-{\rm Id})=({\rm Id} -^tM)Q$ thus ${\rm det}(M-{\rm Id})=-{\rm det}(M-{\rm Id}) \in \Z$. 

%We need a generalization of \cite[Proposition 4.3]{EOY14}:
%\begin{proposition} \label{prop:evenness}
%
%\begin{itemize}
%\item[i)] If $d$ is even, then there is an element $g \in {\rm SO}(L)(\R)$ such that none of the eigenvalues of $g$ is a root of unity;
%
%\item[ii)] If $d$ is odd, then there is an element $g \in {\rm SO}(L)(\R)$ such that the characteristic polynomial of $g$ is equal to $f(t)(t-1)$, where $f(t)\in \R[t]$ is a polynomial such that none of the roots of $f(t)$ is a root of unity.
%\end{itemize}
% \end{proposition}
%
%\begin{proof}
%See the proof of \cite[Proposition 4.3]{EOY14}.
%\end{proof}

We assume $d$ is odd. Let $P_{d} \subset \Z [t]$ be the set of monic polynomials of degree $d$. Then $P_{d}$  is identified 
with  the affine variety $\A^{d}$  defined over $\Z$.
The map  
$${\rm char} : {\rm SO}(L) \to P_{d}\,\, ,\,\, h \mapsto \Phi_{h}(t) := {\rm det} (tI_{d} -h)\,\, $$
is a morphism of affine varieties. 
Let $u_1(t) = t-1, u_2(t) := t+1, \cdots, u_{N}(t)\,\, $  be the cyclotomic polynomials in $\Z[t]$ of degree $\le d$, 
where $N$ is the cardinarity of the cyclotomic polynomials of degree $\le d$. 
The subsets
$P_1 := \{p(t) \in P_{d}(\C)\, | \  u_1^2(t) | p(t) \} \text{ and } P_i := \{p(t) \in P_{d}(\C)\, | \  u_i(t) | p(t) \} \text{ for }i\ge 2$
define
 proper closed algebraic subvarieties of $P_{d}\otimes_{\Z} \Q$, thus so is their finite union 
$Q_{d} := \cup_{i=1}^{N} P_i \subset P_{d}  \otimes_{\Z} \Q   .$ ({\it Notice that the sets $P_i$, $i\ge 2$, defined here is the same  as those defined in \cite{EOY14}. However, the set $P_1$ here is slightly different from that in \cite{EOY14}. This is because, when $d$ is odd, for any $g\in {\rm SO}(L)(\Z)$, $1$ is an eigenvalue of $g$, and hence $t-1$ divides the characteristic polynomial of $g$.})

Let $g \in G$. Its characteristic polynomial $ \Phi_g(t) \in \Z[t]$  is monic and of degree $d$. By Proposition~\ref{salem}, $\Phi_g(t)$ is the product of cyclotomic polynomials and of at most one Salem polynomial counted with multiplicities. Thus, 
$\Phi_g(t)$ is divided by a Salem polynomial of degree $d-1$ if and only if  $ \Phi_g(t) \in P_{d}(\C)\setminus  Q_{d}$. 
The following lemma completes the proof:
\begin{lemma} \label{lem:complete}
There is an element $g \in G$ such that $\Phi_g(t)  \in P_{d}(\C)\setminus Q_{d}$. 
\end{lemma}
\begin{proof} See the proof of \cite[Lemma 4.5]{EOY14}. 
\end{proof}

This completes the proof of Theorem \ref{thm:BlancCantat}.

\end{proof}

\begin{remark}
In order to prove Theorem \ref{thm:main1}, we essentially only need the case i) of Theorem \ref{thm:BlancCantat}.  However, we need both two cases of Theorem \ref{thm:BlancCantat} to prove Theorem \ref{thm:main2} which can apply to other K3 surfaces besides supersingular K3 surfaces.
\end{remark}

\begin{thm} \label{thm:StableSublattice}
Let $L$ be a hyperbolic lattice of signature $(1, r+1)$ with $r \ge 0$ and let $e \in L$ be a primitive element such that $(e, e) = 0$. Let $g \in {\rm SO}(L)(\Z)$ such that ${\rm ord}(g)=\infty$ and  $g(e) = e$. Let $V$ be a  $g$-stable $\R$-linear subspace of $L_{\R}$. Then either  $V$ is in the hypersurface $e^{\perp}$ in $L_{\R}$, or  $e\in V$ (or both).
\end{thm}

\begin{proof}
 Choose a $\Q$-bases of $L_{\Q}$: $$\langle e, w_1, \cdots , w_{d-2}, u \rangle,$$ where $w_i\in e^{\perp}$, and $(u,e)=1$.

By \cite[proof of Lemma 3.6]{Og09} (see also \cite[Lemma 4.7]{EOY14}), replacing $g$ by a suitable power $g^N$ ($N> 0$) if necessary, we may assume $$g= \left(\begin{array}{rrr}
1 & {\mathbf a}^t & c\\
{\mathbf 0} & I_r  & {\mathbf b}\\
0 & {\mathbf 0}^{t} & 1
\end{array} \right)\,\, ,$$ with respect to the $\Q$-bases chosen above. Here $1, 0 \in \Q$ are the unit and the zero, $c$ is in $\Q$, $I_r$ is the $r \times r$ identity matrix, ${\mathbf 0} \in \Q^r$ is the zero vector,  ${\mathbf b} \in \Q^r$ is a column vector,   ${\mathbf a}^t$ 
is the transpose of a column vector ${\mathbf a} \in \Q^r$, and simiarly for ${\mathbf 0}^{t}$.

We claim that ${\mathbf b}\neq {\mathbf 0}$. Suppose otherwise. Then $g(u)=u+ce$, it follows that 
$$(u,u)=(g(u),g(u))=(u+ce, u+ce)=(u,u)+2c(u,e)=(u,u)+2c,$$

which implies $c=0$. Then for all $1\leq i\leq r$ $$(u,w_i)=(g(u),g(w_i))=(u,w_i+a_ie)=(u,w_i)+a_i, {\rm \;where \;\;} a_i={\mathbf a}^t\cdot{\mathbf e_i}.$$ Therefore, $a_i=0$ for all $i$. So $g=Id$, a contradiction to ${\rm ord}(g)=\infty$. Therefore, ${\mathbf b}\neq {\mathbf 0}$.

Without loss of generality, from now on, we may assume ${\mathbf b}={\mathbf e_r}$. 

We claim that $a_r\neq 0$. Suppose otherwise. Then for all $k\geq 1$, $$g^{k}(u)=u+kw_r+kce.$$ Since $g^{k}$ preserves intersection form, it follows that $$(u,u)=(g^{k}(u),g^{k}(u))=(u,u)+k^2(w_r,w_r)+2k(u,w_r)+2kc$$ whence $$(w_r,w_r)k^2+(2(u,w_r)+2c)k=0$$ for all positive integers $k$, a contradiction to $(w_r,w_r)<0$. Therefore, $a_r\neq 0$.

If $V\subset e^{\perp},$ then we are done.

From now on, we assume $V$ is not contained in $e^\perp$. Then there exists $0\neq v\in V$ of the following form $$v=u+\alpha e+\sum_{i=1}^{r}\beta_i w_i,$$ where $\alpha$, $\beta_i\in \R$. Then 
$g(v)=u+w_r+(c+\alpha)e+\sum_{i=1}^{r}\beta_i (w_i+a_ie)$ whence $g(v)-v=w_r+ce+\sum_{i=1}^{r}\beta_i a_ie.$ Then $g(g(v)-v)-(g(v)-v)=a_re\in V$ since $V$ is $g$-stable. It follows that $e\in V$ by $a_r\ne 0$.
\end{proof}

\section{Proof of Theorems \ref{thm:main2}, \ref{thm:main3} and \ref{thm:main1}}\label{ss:ProofofMain}

We need the following result to prove Theorem~\ref{thm:main2}:

\begin{theorem}\label{thm:G}
Let $X$ be a K3 surface defined over an algebraically closed field $k$ of characteristic $p\ne 2,3$. Suppose ${\rm rk}(L_{\infty}(X))\ge 2$, where the sublattice $L_{\infty}(X)\subset {\rm NS}(X)$ is generated by all the elliptic fibrations with infinite automorphism groups. To simplify the notation, we set $L:=L_{\infty}(X)$. Let $d={\rm rk}(L)$. Then there exists a subgroup $G\subset {\rm Aut}(X)$ such that 
\begin{itemize}
\item[1)] $G^{\prime}\subset {\rm SO}^{+}(L)$, where $G^{\prime}:=(\pi\circ\psi)(G)$ (see Section~\ref{ss:elliptic} for definition of $\pi$ and $\psi$), and

\item[2)] Any $G^{\prime}$-stable $\R$-linear subspace of $L_{\R}$ is either $\{0\}$ or $L_{\R}$.
\end{itemize}
\end{theorem}

\begin{proof}
In order to construct $G$, we need the following:
\begin{lemma}\label{lem:g_{e}}
For any $e\in \sE_{\infty}(X)$, there exists $g_{e}\in {\rm Aut}(X)_{e}$ such that  $(\pi\circ \psi )(g_{e})\in {\rm SO}(L)$ and ${\rm ord}((\pi\circ \psi )(g_{e}))=\infty$.
\end{lemma}

\begin{proof}
 By definition of $\sE_{\infty}(X)$, ${\rm Aut}(X)_{e}$ is an infinite subgroup of ${\rm Aut}(X)$. Since, by Lemma \ref{lem:Kerfinite},  ${\rm Ker}(\pi\circ \psi)\cap {\rm Aut}(X)_{e}$ is finite, the image $(\pi\circ \psi) ({\rm Aut}(X)_{e})$ is an infinite subgroup of ${\rm O}(L)$. Thus, there exists $h\in (\pi\circ \psi) ({\rm Aut}(X)_{e})\cap {\rm SO}(L)$ such that ${\rm ord}(h)=\infty$ by \cite[Proposition 2.2 (3)]{Og07} (notice that the proof of \cite[Proposition 2.2 (3)]{Og07} is based on even hyperbolic lattices, and is valid in any characteristic). Choose any $g_{e}\in (\pi\circ \psi)^{-1}(h)$. Then order ${\rm ord}(g_{e})=\infty$ and $(\pi\circ \psi )(g_{e})\in {\rm SO}(L)$ .

\end{proof}

Now, for any $e\in \sE_{\infty}(X)$, by Lemma~\ref{lem:g_{e}}, we can choose some $g_{e}\in {\rm Aut}(X)_{e}$ such that $(\pi\circ \psi )(g_{e})\in {\rm SO}(L)$ and ${\rm ord}((\pi\circ \psi )(g_{e}))=\infty$. We set $G:= $ the subgroup of ${\rm Aut}(X)$ generated by $\{g_{e}|e\in\sE_{\infty}(X)\}$. Let $G^{\prime}:=(\pi\circ\psi)(G)$.  Then $G^{\prime}\subset {\rm SO}(L)$. Since $G^{\prime}(\sE_{\infty}(X))=\sE_{\infty}(X)$, it follows that  $G^{\prime}\subset {\rm SO}^{+}(L)$, which is the statement 1).

Next we prove the statement 2). Let $V$ be a $G^{\prime}$-stable $\R$-linear subspace of $L_{\R}$. We may assume $V\ne L_{\R}$ (otherwise, we are done). Then there exists $e_{0}\in \sE_{\infty}(X)$ such that $e_{0 }\notin V$. Since $V$ is $G^{\prime}$-stable, it follows that $V$ is also $(\pi\circ\psi)(g_{e_{0}})$-stable. Then by Theorem~\ref{thm:StableSublattice},  we have that $V\subset e_{0}^{\bot}$.  For any $e^{\prime} \in \sE_{\infty}(X)$ such that $e^{\prime}\ne e_{0}$, by Hodge-index Theorem, we have that $(e_{0}, e^{\prime})>0$, which implies $e^{\prime}\notin V$. Then by Theorem~\ref{thm:StableSublattice} again, $V\subset e^{\prime \bot}$. Therefore, $V\subset \cap_{e\in\sE_{\infty}(X)}e^{\perp}$.  But $\R\langle \sE_{\infty}(X)\rangle=L_{\R}$ and $L$ is a non-degenerated lattice. Thus $V=0$. This completes the proof of Theorem~\ref{thm:G}.

\end{proof}

%\begin{theorem}\label{thm:ellipticsalem}
%Let $X$ be a K3 surface defined over an algebraically closed field of characteristic $\ne 2, 3$.  Suppose ${\rm rk}(L_{\infty}(X))\ge 2$, where the sublattice $L_{\infty}(X)\subset {\rm NS}(X)$ is generated by all the elliptic fibrations on $X$ with infinite automorphism groups. Let $d={\rm rk}(L_{\infty}(X))$. Then
%
%1) If $d$ is even,  ${\rm max}\{\text{Salem degree of } f|\; f\in {\rm Aut}(X)\}=d$;
%
%2) If $d$ is odd,  ${\rm max}\{\text{Salem degree of } f|\; f\in {\rm Aut}(X)\}=d-1$.
%\end{theorem}

\begin{proof}[Proof of Theorem~\ref{thm:main2}]
We prove the case 1) and the proof for the case 2) is similar.

Suppose $d$ is even. Let $L=L_{\infty}(X)$. Since $d\ge 2$, it follows that the orthogonal complement $L^{\perp}$ of $L$ in ${\rm NS}(X)$ is negative definite. Note that, for any $f\in {\rm Aut}(X)$, we have $f^{*}(L)=L$ and $f^{*}(L^{\perp})=L^{\perp}$. So the Salem degree of any automorphism of $X$ is an even integer $\leq d$. Thus, ${\rm max}\{\text{Salem degree of } f|\; f\in {\rm Aut}(X)\}\le d$. On the other hand, by Theorem~\ref{thm:BlancCantat} and Theorem~\ref{thm:G}, $X$ has an automorphism the entropy of which is a Salem number of degree $d$, which implies ${\rm max}\{\text{Salem degree of } f|\; f\in {\rm Aut}(X)\}\ge d$. Therefore, ${\rm max}\{\text{Salem degree of } f|\; f\in {\rm Aut}(X)\}=d$.
\end{proof}

%\begin{theorem}\label{thm:compare}
%Let $X$ and $Y$ be two K3 surfaces defined over two algebraically closed fields $k$ and $k^{\prime}$ (${\rm char}(k),{\rm char}(k^{\prime})\ne 2,3$) of the same Picard number $\rho$. Suppose the Neron-Severi group $S_{X}$ of $X$ is isometric to a sublattice of the Neron-Severi group $S_{Y}$ of $Y$, and suppose $Y$ has at least two different elliptic fibrations with infinite automorphism groups. Then $${\rm max}\{\text{Salem degree of } f|\; f\in {\rm Aut}(X)\}\ge {\rm max}\{\text{Salem degree of } f|\; f\in {\rm Aut}(Y)\}.$$
%\end{theorem}

\begin{proof}[Proof of Theorem~\ref{thm:main3}]
If $Y$ has exactly one elliptic fibration with infinite automorphism group (i.e., ${\rm rk}(L_{\infty}(Y))=1$), then every automorphism of $Y$ must preserve this elliptic fibration. Then, by \cite[Proposition~2.9]{Og07} (again the proof of \cite[Proposition~2.9]{Og07} is valid in any characteristic), the entropy of any automorphism of $Y$ must be zero, which implies ${\rm max}\{\text{Salem degree of } f|\; f\in {\rm Aut}(Y)\}=0$ (thus the conclusion of Theorem~\ref{thm:main3} is true).  

Therefore,  we may assume $Y$ has at least two elliptic fibrations with infinite automorphism groups. Then by Theorem~\ref{thm:main2}, it suffices to prove ${\rm rk}(L_{\infty}(X))\ge {\rm rk}(L_{\infty}(Y))$. 

Fix an isometric embedding $\iota : {\rm NS}(X) \hookrightarrow {\rm NS}(Y)$. Let $d={\rm rk}(L_{\infty}(Y))$. Then we can choose $e_{1},...e_{d}\in \sE_{\infty}(Y)$ such that $\langle e_{1},..., e_{d}\rangle$ forms a $\Q$-basis of $L_{\infty}(Y)\otimes \Q$.

Let $h_{Y}\in {\rm Amp}(Y)\cap {\rm NS}(Y)$ be an ample class.  For any $c\in {\rm NS}(Y)$ of square $-2$, by Riemann-Roch Theorem, either $c$ or $-c$ is effective. Thus,  the intersection pairing between $h_{Y}$ and any class in ${\rm NS}(Y)$ of square $-2$ is not zero. Since $\iota ({\rm NS}(X))$ is a sublattice of ${\rm NS}(Y)$ of finite index, there exists a sufficiently large integer $N>0$, such that $Ne_{1},...,Ne_{d},Nh_{Y}\in \iota({\rm NS}(X))$.  Note that the intersection pairing between $Nh_{Y}$ and any class in $\iota({\rm NS}(X))$ of square $-2$ is not zero. Since the ample cone ${\rm Amp}(X)$ of $X$ is a standard fundamental domain for the Weyl group $W({\rm NS}(X))$, it follows that there exists $\alpha \in W({\rm NS}(X))$ such that $\alpha (\iota^{-1} (Nh_{Y}))\in {\rm Amp}(X)\cap {\rm NS}(X)$. Then we claim the following:

\begin{lemma}
 $\alpha (\iota^{-1} (Ne_{1})),..., \alpha (\iota^{-1} (Ne_{d}))\in {\rm Nef} (X)\cap {\rm NS}(X)$, where ${\rm Nef} (X)$ is the nef cone of $X$. 
\end{lemma}

\begin{proof}

Let $C\subset X$ be a smooth rational curve. Then $(\alpha (\iota^{-1} (Nh_{Y})), [C])>0$ by ampleness of $\alpha (\iota^{-1} (Nh_{Y}))$, where $[C]$ denotes the class of $C$ in ${\rm NS}(X)$. Then $(Nh_{Y},\iota(\alpha^{-1}([C])))>0$. Since $\iota(\alpha^{-1}([C]))$ of square $-2$, by Riemann-Roch Theorem, $\iota(\alpha^{-1}([C]))$ is an effective class in $ {\rm NS}(Y)$. Then $(Ne_{i}, \iota(\alpha^{-1}([C])))\ge 0$, for all $i$. Thus, $(\alpha(\iota^{-1}(Ne_{i})),[C])\ge 0$. So $\alpha(\iota^{-1}(Ne_{i}))$ is a nef class in ${\rm NS}(X)$.

\end{proof}

For any $1\le i\le d$, let $e_{i}^{\prime}$ be the primitive class in ${\rm NS}(X)$ such that $\R^{>0}e_{i}^{\prime}=\R^{>0}\alpha(\iota^{-1}(Ne_{i}))$. Then $e_{i}^{\prime}\in \sE(X)$ according to Piatetsky-Shapiro and Shafarevich \cite{PS71}. 

 \begin{lemma}
 $e_{i}^{\prime}\in \sE_{\infty}(X)$ for all $i$.
 \end{lemma}
  
  \begin{proof}
  Since $e_{i}\in \sE_{\infty}(Y)$, by Lemma~\ref{lem:einf}, ${\rm rk}(e_i^{\perp})-{\rm rk}((e_i^{\perp})^{(2)})>0$. Since ${\rm rk}(${\rm NS}(X)$)={\rm rk}(${\rm NS}(Y)$)$, it follows that ${\rm rk}(e_i^{\perp})={\rm rk}(e_i^{\prime \perp})$. Since ${\rm NS}(X)$ is isometric to a sublattice of ${\rm NS}(Y)$, it follows that $(e_i^{\prime \perp})^{(2)}$ is isometric to a sublattice of $(e_i^{\perp})^{(2)}$. Then ${\rm rk}((e_i^{\prime \perp})^{(2)})\le {\rm rk}((e_i^{\perp})^{(2)})$. Thus,  ${\rm rk}(e_i^{\prime \perp})-{\rm rk}((e_i^{\prime\perp})^{(2)})\ge  {\rm rk}(e_i^{\perp})-{\rm rk}((e_i^{\perp})^{(2)})>0$. Then, again by Lemma~\ref{lem:einf}, $e_{i}^{\prime}\in \sE_{\infty}(X)$.
  \end{proof}

     Since ${\rm dim}_{\Q}(\Q \langle e_{1},...,e_{d}\rangle)=d$, it follows that ${\rm dim}_{\Q}(\Q\langle e_{1}^{\prime},...,e_{d}^{\prime}\rangle)=d$. Thus, ${\rm rk}(L_{\infty}(X))\ge d={\rm rk}(L_{\infty}(Y))$. Then by Theorem~\ref{thm:main2},$${\rm max}\{\text{Salem degree of } f|\; f\in {\rm Aut}(X)\}\ge {\rm max}\{\text{Salem degree of } f|\; f\in {\rm Aut}(Y)\}.$$ This completes the proof of Theorem~\ref{thm:main3}.
\end{proof}

%\begin{theorem}\label{thm:supersingular}
%Let $p$ be an odd prime. Let $X$ be a supersingular K3 surface defined over an algebraically closed field of characteristic $p$. Then there is an automorphism $f\in Aut(X)$ the entropy of which is the logarithm of a Salem number of degree 22.
%\end{theorem}

\begin{proof}[Proof of Theorem~\ref{thm:main1}]
First we consider the cases $p=11$ or $>13$. Let $X(p)\cong {\rm Km}(E\times_{\F_{p}}E)$ be the unique supersingular K3 surface of Artin invariant one. The two natural projections from $E\times_{\F_{p}} E$ to the two factors induce two elliptic fibrations on $X(p)$ with Mordell-Weil rank $4$,  by the formula of Mordell-Weil rank \cite{Shi90}. Hence ${\rm rk}(L_{\infty}(X(p)))\ge 2$. By Theorem~\ref{thm:known},  ${\rm max}\{\text{Salem degree of } f|\; f\in {\rm Aut}(X(p))\}=22$. By \cite{RS78}, the Artin invariant $\sigma (X)$ of $X$ determines  ${\rm NS}(X)$ up to isometry. Moreover, by explicit classification of the lattices ${\rm NS}(X)$ and ${\rm NS}(X(p))$ in \cite{RS78},  ${\rm NS}(X)$ is isometric to a sublattice of ${\rm NS}(X(p))$ (cf.~\cite[proof of Proposition 3.9]{Li15}). Then by Theorem~\ref{thm:main3}, ${\rm max}\{\text{Salem degree of } f|\; f\in {\rm Aut}(X)\}\ge {\rm max}\{\text{Salem degree of } f|\; f\in {\rm Aut}(X(p))\}=22$. Then the maximal Salem degree of automorphisms of $X$ is $22$ since the Picard number of $X$ is $22$. Thus, when $p=11$ or $>13$, there exists $f\in {\rm Aut}(X)$ the entropy of which is the logarithm of a Salem number of degree $22$. The case $p=3$ is proved by \cite{EO15} and \cite{Sh16}. The cases $p=5,7,13$ are proved by \cite{Sh16}. This completes the proof of Theorem~\ref{thm:main1}.

\end{proof}

\section{The exceptional sublattice, elliptic fibrations and  Salem numbers}\label{ss:E}

In this section, we discuss some relationships among the exceptional sublattice, elliptic fibration and the maximal Salem degree of automorphisms of K3 surfaces. 

 Following \cite{Ni14}, for a K3 surface $X$, we define the {\it exceptional sublattice} of the N\'eron-Severi group ${\rm NS}(X)$ by $$E({\rm NS}(X)):=\{x\in {\rm NS}(X)| \text{ the orbit } {\rm Aut}(X)(x) \text{ of } x \text{ in }{\rm NS}(X)  \text{ is finite} \}.$$ Clearly, $E({\rm NS}(X))$ is a primitive sublattice of ${\rm NS}(X)$.  For a sublattice $F\subset {\rm NS}(X)$ we denote by $F_{pr}$ the primitive sublattice $F_{pr}={\rm NS}(X)\cap (F\otimes \Q)\subset {\rm NS}(X)\otimes \Q$ generated by $F$. 

\begin{theorem}\label{thm:E}
(\cite[Theorem 4.1]{Ni14}) Let $X$ be a K3 surface defined over an algebraically closed field $k$ of characteristic $p\ne 2,3$. Suppose $X$ has at least two elliptic fibrations with infinite automorphism groups. Then $$E({\rm NS}(X))=\underset{e\in \sE_{\infty}(X)}{\cap}(e^{\perp})^{(2)}_{pr}=L_{\infty}(X)^{\perp}.$$
\end{theorem}

\begin{proof}
Since $X$  has at least two elliptic fibrations with infinite automorphism groups, it follows that the sublattice $L_{\infty}(X)^{\perp}\subset {\rm NS}(X)$ is negative definite. Then  $L_{\infty}(X)^{\perp}\subset E({\rm NS}(X))$ since $L_{\infty}(X)^{\perp}$ is ${\rm Aut}(X)$-stable and ${\rm O}(L_{\infty}(X)^{\perp})$ is a finite group. On the other hand, by \cite[Theorem 4.1]{Ni14}, $E({\rm NS}(X))=\underset{e\in \sE_{\infty}(X)}{\cap}(e^{\perp})^{(2)}_{pr}\subset L_{\infty}(X)^{\perp}$.  Thus, $E({\rm NS}(X))=\underset{e\in \sE_{\infty}(X)}{\cap}(e^{\perp})^{(2)}_{pr}=L_{\infty}(X)^{\perp}.$
\end{proof}

For a K3 surface with even Picard number, we collect various methods to check whether it has an automorphism of maximal possible Salem degree:
   
\begin{theorem}\label{thm:evenPicard}
Let $X$ be a K3 surface defined over an algebraically closed field $k$ of characteristic $p\ne 2,3$. Suppose $X$ has even Picard number $\rho (X)\ge 4$, and suppose $X$ has at least one elliptic fibration with infinite automorphism group. Then the following statements are equivalent to each other:
\begin{itemize}
\item[1)] Any ${\rm Aut}(X)$-stable $\R$-linear subspace of ${\rm NS}(X)\otimes \R$ is either $\{0\}$ or ${\rm NS}(X)\otimes \R$;

\item[2)] Any ${\rm Aut}(X)$-stable $\Q$-linear subspace of ${\rm NS}(X)\otimes \Q$ is either $\{0\}$ or ${\rm NS}(X)\otimes \Q$;

\item[3)] $\Q \langle \sE_{\infty}(X)\rangle= {\rm NS}(X)\otimes \Q$;

\item[4)] There exists $f\in {\rm Aut}(X)$ such that the Salem degree of $f$ is $\rho (X )$;

\item[5)] $\underset{e\in \sE_{\infty}(X)}{\cap}(e^{\perp})^{(2)}_{pr}=\{0\}$;

\item[6)] $E({\rm NS}(X))=\{0\}$.
\end{itemize}
\end{theorem}

\begin{proof}
1) $\Longrightarrow$ 2) $\Longrightarrow$ 3): Trivial.

3) $\Longrightarrow$ 1): By Theorem~\ref{thm:G}.

3) $\Longrightarrow$ 4): By Theorem~\ref{thm:main2}.

4) $\Longrightarrow$ 3): Suppose $X$ has an automorphism $f\in {\rm Aut}(X)$ whose entropy is the logarithm of a Salem number of degree $\rho(X)$. Then ${\rm NS}(X)\otimes\Q$ has no non-trivial $f$-stable $\Q$-subspace (cf. \cite[Proof of Theorem 3.4]{Mc02}). On the other hand, $\Q\langle\sE_{\infty}(X)\rangle$ is clearly $f$-stable. Therefore, $\Q \langle\sE_{\infty}(X)\rangle={\rm NS}(X)\otimes\Q$.

5) $\Longleftrightarrow$ 6): By Theorem~\ref{thm:E}.

2) $\Longrightarrow$ 6): By definition, $E({\rm NS}(X))\otimes \Q$ is  clearly an ${\rm Aut}(X)$-stable $\Q$-linear subspace of ${\rm NS}(X)\otimes \Q$. Then either $E({\rm NS}(X))=\{0\}$ or ${\rm NS}(X)$. If $E({\rm NS}(X))={\rm NS}(X)$, then ${\rm Aut}(X)$ must be a finite group, a contradiction to the assumption $\sE_{\infty}(X)\ne \emptyset$. Therefore, $E({\rm NS}(X))=\{0\}$.

6) $\Longrightarrow$ 3): Since $E({\rm NS}(X))=\{0\}$, it follows that $\sE_{\infty}(X)$ is an infinite set. Then the sublattice $L_{\infty}(X)^{\perp}\subset {\rm NS}(X)$ is negative definite and ${\rm O}(L_{\infty}(X)^{\perp})$ is a finite group. Obviously, $L_{\infty}(X)^{\perp}$ is ${\rm Aut}(X)$-stable. Then  $L_{\infty}(X)^{\perp}\subset E({\rm NS}(X))=\{0\}$. Thus, $\Q \langle \sE_{\infty}(X)\rangle= {\rm NS}(X)\otimes \Q$.

\end{proof}

\begin{remark}
\begin{itemize}

\item[i)] By Theroems~\ref{thm:main1} and \ref{thm:evenPicard}, $E({\rm NS}(X))=\{0\}$ for any supersingular K3 surface $X$ in odd characteristic.

\item[ii)] Let $X$ be the Kummer surface of the Jacobian of a very general curve $C$ of genus $2$ over an algebraically closed field of characteristic $0$. Then $X$ is a K3 surface of Picard number $17$. Elliptic fibrations with a section on $X$ are completely classified in \cite{Ku14}. Using elliptic fibrations explicitly given in the table on \cite[p.~609-610]{Ku14}, one can easily show that ${\rm rk}(L_{\infty}(X))=17$. Thus, by Theorem~\ref{thm:main2}, $X$ has an automorphism the entropy of which is the logarithm of a Salem number of degree $16$.

\item[iii)] Let $X$ be the Kummer surface ${\rm Km}(E\times F)$, where complex elliptic curves $E$ and $F$ are not isogenous. Then $X$ is a complex K3 surface with Picard number $18$. Thanks to complete classification of elliptic fibrations with a section on $X$ (\cite{Og89}), one can show that ${\rm rk}(L_{\infty}(X))\ge 10$.  Thus, by Theorem~\ref{thm:main2}, ${\rm max}\{\text{Salem degree of } f|\; f\in {\rm Aut}(X)\}\ge 10$. (\cite{Og16} proves that ${\rm max}\{\text{Salem degree of } f|\; f\in {\rm Aut}(X)\}\ge 6$ using a different approach.) On the other hand, the exceptional lattice $E({\rm NS}(X))$ is negative definite and, by \cite[Lemma~1.4]{Og89}, contains a sublattice of rank $8$. Then the Salem degree of any automorphism of $X$ must be  $\le18-8=10$. Therefore, we conclude that ${\rm max}\{\text{Salem degree of } f|\; f\in {\rm Aut}(X)\}= 10$. Note that ${\rm NS}(X)$ is a 2-elementary lattice (for similar examples of K3 surfaces $Y$ with $E({\rm NS}(Y))\ne \{0\}$, see \cite[Section~3]{Ni99}). 

\end{itemize}
\end{remark}

\begin{remark}
 Let $\sS {\rm EK}3^{\prime}$ be the set of all even hyperbolic lattices $S$ of ${\rm rk}(S)\ge 3$ with the following property:  {\it There exists a K3 surface $X$ defined over an algebraically closed field of characteristic $\ne 2,3$ such that $S$ is isometric to ${\rm NS}(X)$, and $E({\rm NS}(X))\ne \{0\}$.} By \cite[Theorem~4.4]{Ni14} and \cite[Theorem~6.1]{LM11}, the set $\sS {\rm EK}3^{\prime}$ is finite (note that $\sS {\rm EK}3^{\prime}$ defined here is  a subset of the set $\sS {\rm EK}3$ defined by \cite[Definition~4.5]{Ni14}). A consequence of finiteness of $\sS {\rm EK}3^{\prime}$ is the following: There are only finitely many {\it singular} K3 surfaces $X$ (i.e., complex K3 surfaces with Picard number $20$) such that ${\rm max}\{\text{Salem degree of } f|\; f\in {\rm Aut}(X)\}\le 18$. It would be interesting to find the  finite set of N\'eron-Severi groups $\sS {\rm EK}3^{\prime}$ (or even $\sS {\rm EK}3$) of K3 surfaces.
 
 \end{remark}

\end{document}